\documentclass{cimart}

\usepackage{latexsym}
\usepackage{amsfonts}
\usepackage{amsmath, amssymb}
\usepackage{bbm}
\usepackage{color}
\usepackage{tikz}
\usepackage{cite}
\usepackage{enumerate}
\usepackage{graphicx}

\newcommand{\ZZ}{\mathbb{Z}}
\newcommand{\RR}{\mathbb{R}}
\newcommand{\CC}{\mathbb{C}}
\newcommand{\QQ}{\mathbb{Q}}

\DeclareMathOperator{\GCD}{GCD}

\DeclareMathOperator{\OO}{O}
\DeclareMathOperator{\SO}{SO}
\DeclareMathOperator{\Reup}{Re}

\renewcommand{\Re}{\Reup}

\newcommand{\e}[1]{\exp \left (#1 \right )}

\title{% Please, capitalize only the first word
    Symmetry groups and deformations of sums of exponentials
    }

\author{% Please, use "Firstname Lastname" format, without abbreviations
    Florian Pausinger and David Petrecca
    }

\authorinfo[% Format: please, use "F. Name"
    F. Pausinger]{% Format: please, use "Department (only if 
    DM and CMAFcIO, Faculty of Sciences, University of Lisbon, Portugal}{%
    fpausinger@fc.ul.pt
    }

\authorinfo[% Format: please, use "S. Name"
    D. Petrecca]{% Format: please, use "Department (only if 
    Stuttgart, Germany}{%
    david.petrecca@gmail.com
    }

\abstract{%
We study the symmetry groups and winding numbers of planar curves obtained as images of weighted sums of exponentials. More generally, we study the image of the complex unit circle under a finite or infinite Laurent series using a particular parametrisation of the circle. We generalise various previous results on such sums of exponentials and relate them to other classes of curves present in the literature.
Moreover, we consider the evolution under the wave equation of such curves for the case of binomials. Interestingly, our methods provide a unified and systematic way of constructing curves with prescribed properties, such as the number of cusps, the number of intersection points or the winding number.
    }



\VOLUME{33}
\YEAR{2025}
\ISSUE{1}
\NUMBER{4}
\DOI{https://doi.org/10.46298/cm.13932}

\begin{document}

% Here is where the main text should be typed:

% A table of contents will be automatically inserted in your article if it
% has 3 or more sections.  Please, do not try to manually change this
% behaviour.

% Also, please consider the following suggestions while preparing your 
% manuscript (as they will speed up the editorial process):
% * Avoid starting a new sentence with a mathematical formula;
% * Try to separate adjacent formulas with words;
% * Avoid inline formulas longer than half of a line. You can use math 
%   displays (\[...\]) instead;
% * Consider the use the enumerate and itemize environments for lists;
% * Consider the use of \dots, \ldots, \dotsc, \cdot, etc, instead of "..." 
%   or ".";
% * Instead of numbering or citing an article by hand (using parenthesis or 
%   brackets), consider the use of \cite, \ref and \eqref for citations and
%   cross-references;
% * Try to avoid inserting horizontal or vertical spacing, such as \hskip, 
%   \vskip and \bigskip;
% * Try to avoid inserting line or page brakes, such as \\, \newpage and
%   \clearpage.

% Acknowledgments should be added at the end of this section (right before
% the refences section) as a \subsection* (a subsection without a number):
% \subsection*{Acknowledgments} ...

\section{Introduction}
\subsection{Context}
In the late nineties, Farris~\cite{farris1996wheels} considered curves obtained through a superposition of circular motions. 
A weighted sum of exponentials 
$$\sum_k w_k \e{2 \pi i a_k t}$$ 
can be used to describe such a construction in which the $a_k$ are real exponents and the weights $w_k$ are positive.

Farris noticed the symmetry of such curves, which motivated his later book \cite{farris2015creating}, where he introduces the concept of a \emph{rosette curve}. A rosette curve is a planar periodic curve admitting dihedral or rotational symmetries and is obtained from the image of the unit circle under a Laurent polynomial\footnote{This should not be confused with other definitions of rosette present in the literature that require convexity and/or positivity of the curvature as planar differentiable curves.}.

Several well-known curves can be included in this framework, such as roses or Rhodonea curves~\cite{diff2022}.
They are defined as the set of points in polar coordinates $(\rho, \theta)$ that satisfy the equation $\rho = \sin(n \theta)$ which can be seen as the image of the unit circle $t \mapsto \e{2 \pi i t}$ under the Laurent polynomial
\[
    f(z) = \frac{z^{1+n}- z^{1-n}}{2i}.
    \] 
This framework allows to neatly rephrase some open questions of Maurer~\cite{maurer_roses}. The authors plan to address them in a future paper. 

Interestingly, the curves that can be drawn using the Spirograph\texttrademark \ \  device can also be included in this framework.
Such curves are the trajectories of a point on a toothed wheel of radius $r$ rolling along the interior or exterior of another of radius $R>r.$ After some computations the trajectory boils down to a weighted sum of exponentials that can be seen as the image of the circle under the polynomial
\[
    p(z) = (R-r)z^r + r z^{R-r}\] 
if the two radii are positive integers. 
Hence, the superposition of circular motions is closely related to a wheel-rolling-on-another-wheel construction as both can be described using weighted sums of exponentials.

In two recent articles \cite{pausinger2021symmetry, pausinger2023symmetry} the first author together with Vartziotis studied the planar curves obtained by the sum of two complex exponentials. In particular, results about the symmetry groups of their graph, a description of the self-intersections as well as some differential geometric properties such as winding numbers and presence of cusps were obtained. Several questions were left open, such as proper generalizations to arbitrary sums of exponentials or the study of the case with irrational exponentials.

\subsection{Results and outlook} 
In this paper, we significantly extend the work from \cite{pausinger2021symmetry, pausinger2023symmetry} and explore the connections with the rosette notion introduced by Farris.

We generalise and extend the results to Laurent series $f$ defined on the complex plane, including weighted sums of exponentials and complex polynomials. We study the image of the unit circle under a Laurent series or one of the above types of functions. In particular, we study the symmetries of the image $f(S^1)$ and prove they carry over as permutations of the zeros and poles of $f.$ We use a notion of symmetry introduced by Farris \cite{farris2015creating} and relate our results to the results of Quine \cite{quine1976geometry} for the study of self-intersections. 

In more detail, we answer in Theorem~\ref{thm:symmetry} and Prop.~\ref{prop:symmtype} the question stated in~\cite{pausinger2021symmetry} about the symmetry group of the image of the unit circle through a Laurent series. Then we exhibit examples of symmetries or lack thereof in the case of polynomials, and finally, we study the case of irrational exponents in Theorem \ref{thm:sym1} and Theorem \ref{thm:dense_annulus} answering another question stated in~\cite{pausinger2023symmetry}. In the context of the ``wheels\ldots on wheels'' mechanical interpretation, this is an analogue of the dense image of Lissajous curves in the case of incommensurable frequencies.

In addition, and following an idea of Farris \cite{farris2015creating}, we investigate the evolution of our curves under the wave equation in Section \ref{sec:wave}. To simplify calculations, we restrict to the case of polynomials with two terms and study the variation of the winding number along such flow. We provide in Theorem~\ref{thm:winding2terms} an analog to a winding number result proven in~\cite{pausinger2023symmetry} in a different setting. Moreover, at times during the evolution for which the polynomial coefficients are not integral, self-intersections and cusp formations may occur. Various examples of this phenomenon are investigated in Section \ref{sec:self} and in particular in Theorems \ref{thm:selfintersections} and \ref{thm:selfintersections_wave}.

%%%%%%%%%%%%%%%%%%%%%%
%%%%%%%%%%%%%%%%%%%%%%%
\section{Symmetry groups}

Farris~\cite{farris2015creating} gives the following definition.
\begin{definition}\label{def:symm}
Let $k < m$ be a pair of coprime positive integers. A function $f \colon \RR \to \CC$ periodic of period $1$ has symmetry of type $(k, m)$ if
\begin{equation} \label{eq:symkm}
    f \biggl (t + \frac 1 m \biggr ) = \e{ 2 \pi i \frac k m} f(t)
\end{equation}
for all $t \in \RR.$ Moreover, $f$ has mirror symmetry along the axis generated by the line $\RR \e{i \pi \sigma}$ if
\begin{equation*} %\label{eq:symmirror}
    f(-t) = \e{2 \pi i \sigma} \overline{f(t)}
\end{equation*}
for all $t \in \RR.$
\end{definition}
In what follows, we consider $f$ to be a Laurent series on the complex plane, i.e. 
$$f(z) = \sum_{n \in \ZZ} c_n z^n$$
and, if not otherwise stated, the sums will range over $n \in \ZZ.$
Definition~\ref{def:symm} can be generalised to such setting as follows.
\begin{definition}
    \label{def:symm_laurent}
    Let $k < m$ be a pair of coprime positive integers. A Laurent series $f$ defined on the complex plane $\CC$ has symmetry of type $(k, m)$ if
\begin{equation*} %\label{eq:symkm_laurent}
    f \biggl (\e{ 2 \pi i \frac 1 m} z \biggr ) = \e{ 2 \pi i \frac k m} f(z)
\end{equation*}
for all $z \in \CC.$ Moreover, $f$ has mirror symmetry along the axis generated by the line $\RR \e{i \sigma}$ if
\begin{equation*} %\label{eq:symmirror_laurent}
    f(\overline z) = \e{2 \pi i \sigma} \overline{f(z)}
\end{equation*}
for all $z \in \CC.$
\end{definition}
We consider the question of how symmetries relate to planar isometries that leave the image under $f$ of the unit circle $S^1 = \{ z \in \CC: |z|=1 \}$ invariant.
\begin{definition}[\cite{pausinger2021symmetry}]
    A subgroup $G \subseteq \OO(2)$ is a symmetry group for $f(S^1)$ if $G (f(S^1)) \subseteq f(S^1).$
\end{definition} 
We immediately see that a symmetry group contains all the symmetries from Definition~\ref{def:symm}.
\begin{proposition}\label{prop:symmtype}
    If $f$ has symmetry of type $(k, m)$ then the isometry group of $f(S^1)$ contains the cyclic group of order $m.$ Moreover, if $f$ also has a mirror symmetry along any axis, then $G$ contains the dihedral group $D_{m}$ of order $2m.$
\end{proposition}
\begin{proof}
    Since $\GCD(k, m)=1$, we have that $g=\e{2 \pi i k/m}$ generates the cyclic group of order $m$, so the first inclusion follows from ~\eqref{eq:symkm}. If $f$ is also mirror symmetric, then we see that $G$ also must contain the map $ \tau \colon z \mapsto \e{2 \pi i \sigma} \overline z.$
    The conclusion follows from the fact that $g$ and $\tau$ generate $D_{m}.$
\end{proof}

\begin{remark}
If the Laurent polynomial $f$ has symmetry of type $(k, m)$, the isometry group of $f(S^1)$ can strictly contain the cyclic group of order $m.$ Indeed, consider the Laurent polynomial that describes a Rhodonea of four petals
\begin{equation*} %\label{eq:rhodonea4}
    f(z) = \frac 1 {2i}(z^3 - z^{-1}).
\end{equation*}
According to Definition~\ref{def:symm}, $f$ has symmetry of type $(1, 2)$ since all exponents are congruent to $1$ modulo $2$ but the isometry group of $f(S^1)$ is strictly larger than the group of order two. In fact, $f$ has also symmetry of type $(3, 4)$ for an analogous reason.
\end{remark}
Kronecker's Theorem together with the fact that $\SO(2)$ is isomorphic to the circle group implies that every subgroup of the planar isometry group $\OO(2)$ is either finite or dense. In the latter case, it is generated by a rotation of irrational angle $g\colon z \mapsto \e{ 2 \pi i \alpha z}.$
The following lemma explores the relation between $\alpha$ and the symmetries of $f.$ Any two of such symmetries are related by the following lemma.

\begin{lemma}\label{lemma:rational_angle}
    Let $f(z) = \sum_n c_n z^n$ be a Laurent series and let $\alpha \in \RR$ be such that $$\e{2 \pi i \alpha} f(z) = f( \e{2 \pi i \beta}z)$$ for some $\beta \in \RR.$ Then $\alpha$ and $\beta$ are rationally dependent. In particular, if $f$ has at least two terms then $\alpha, \beta \in \QQ.$ 
    Moreover, a mirror symmetry $z \mapsto \e{2 \pi i \sigma} \overline z$ is a symmetry of $f(S^1)$ if and only if $\exp(2 \pi i \sigma)\overline{c_n} =c_n.$
\end{lemma}
\begin{proof}For every $n$ such that $c_n \neq 0$ we must have that the coefficient of $z^n$ satisfies $c_n\e{ 2 \pi i \alpha} = c_n\e{2 \pi i n \beta}.$ This means that $\alpha-n \beta \in \ZZ$ so $\alpha$ and $\beta$ are rationally dependent.
    If $n_1 \neq n_2$ are such that $c_{n_1} \neq 0$ and $c_{n_2} \neq 0$, by subtracting we get that
    $(n_1 - n_2) \beta \in \ZZ$, so $\beta$ must be rational and hence also $\alpha.$
    If $g \colon z \mapsto \e{2 \pi i \sigma \overline z}$ is a mirror symmetry that leaves $f$ invariant, one has that
    \[
        \sum_n c_n \overline z^n = f(\overline z) = \e{2 \pi i \sigma} \sum_n \overline{c_n} \overline z^n 
    \]
    so, by equating coefficients, we must have for every $n$ either $c_n = 0$ or $c_n = \e{2 \pi i \sigma} \overline{c_n}.$
\end{proof}
Lemma~\ref{lemma:rational_angle} allows us to obtain the following theorem.
\begin{theorem} \label{thm:symmetry}
    Let $f$ be a Laurent polynomial of at least two terms on the complex plane. Then the symmetry group $G$ of $f(S^1)$ is either trivial, cyclic or dihedral.
\end{theorem}
    \begin{proof}
    Let $g$ be a rotational symmetry of $f$, i.e. $g(f(S^1)) =f(S^1).$ By Lemma~\ref{lemma:rational_angle}, it must be the rotation by a rational multiple $p/q$ of $\pi$, with $\GCD(p, q)=1.$ So $g$ generates a cyclic group of order $q$ and hence we have that $\ZZ/ q\ZZ \subseteq G.$ If $G$ also contains a mirror symmetry, this would together with $g$ generate a dihedral group of order $2q$, so $D_{q} \subseteq G.$
    \end{proof}
\begin{remark}
Let us step back to complex polynomials and use the notations of~\cite{pausinger2021symmetry,pausinger2023symmetry}.
    Consider $p(z) = \sum w_k z^{a_k}$, where $a_1 < a_2 < \ldots a_n$ and let 
    $$m = \GCD(\{a_k - a_j: 1 \leq j < k \leq n\}).$$ Then $p$ has always mirror symmetry with $\sigma=0$ if $w_k \in \RR$ by Lemma~\ref{lemma:rational_angle} and has symmetry of type $(a_1, m)$ if $\GCD(a_1, m)=1$ since
\begin{eqnarray*}
p\biggl (\exp \biggl (2 \pi i \frac{a_1}{m} \biggr )z \biggr ) &=& \exp \biggl (2 \pi i \frac{a_1}{m} \biggr ) z^{a_1}\sum_{k=1}^n w_k \exp \bigg ( 2 \pi i \frac{a_k-a_1}{m} \biggr ) z^{a_k-a_1}\\
                                                               &=& \exp \biggl (2 \pi i \frac{a_1}{m} \biggr ) z^{a_1}\sum_{k=1}^n w_k z^{a_k-a_1}    
\end{eqnarray*}
where the last equality holds since $m$ divides $a_k-a_1$ for all $k.$
We notice as well that this computation does not consider the polynomial coefficients $w_k$ at all, meaning that they do not affect the type of symmetry $(a_1, m)$ of $p.$ This answers the question stated in the cited papers about the symmetry groups of an arbitrary weighted sum of exponentials.
\end{remark}
\begin{example}
For all $a, b, c \in \CC$, the polynomials $p(z) = a z^2 + b z^7 + c z^{12}$ have symmetry of type $(2, 5)$ since the smallest exponent is $2$ and is coprime with $\GCD(5, 5, 10) = 5.$ The symmetry group of $p(S^1)$ can contain $D_{5}$ or just $\ZZ /5 \ZZ$ as can be seen in Figure~\ref{fig:symmetry25}. Note that the polynomial shown on the right of Figure~\ref{fig:symmetry25} has some complex coefficients in contrast to the polynomial shown on the left.
\end{example}
\begin{figure}
\begin{center}
    \includegraphics[scale=0.2]{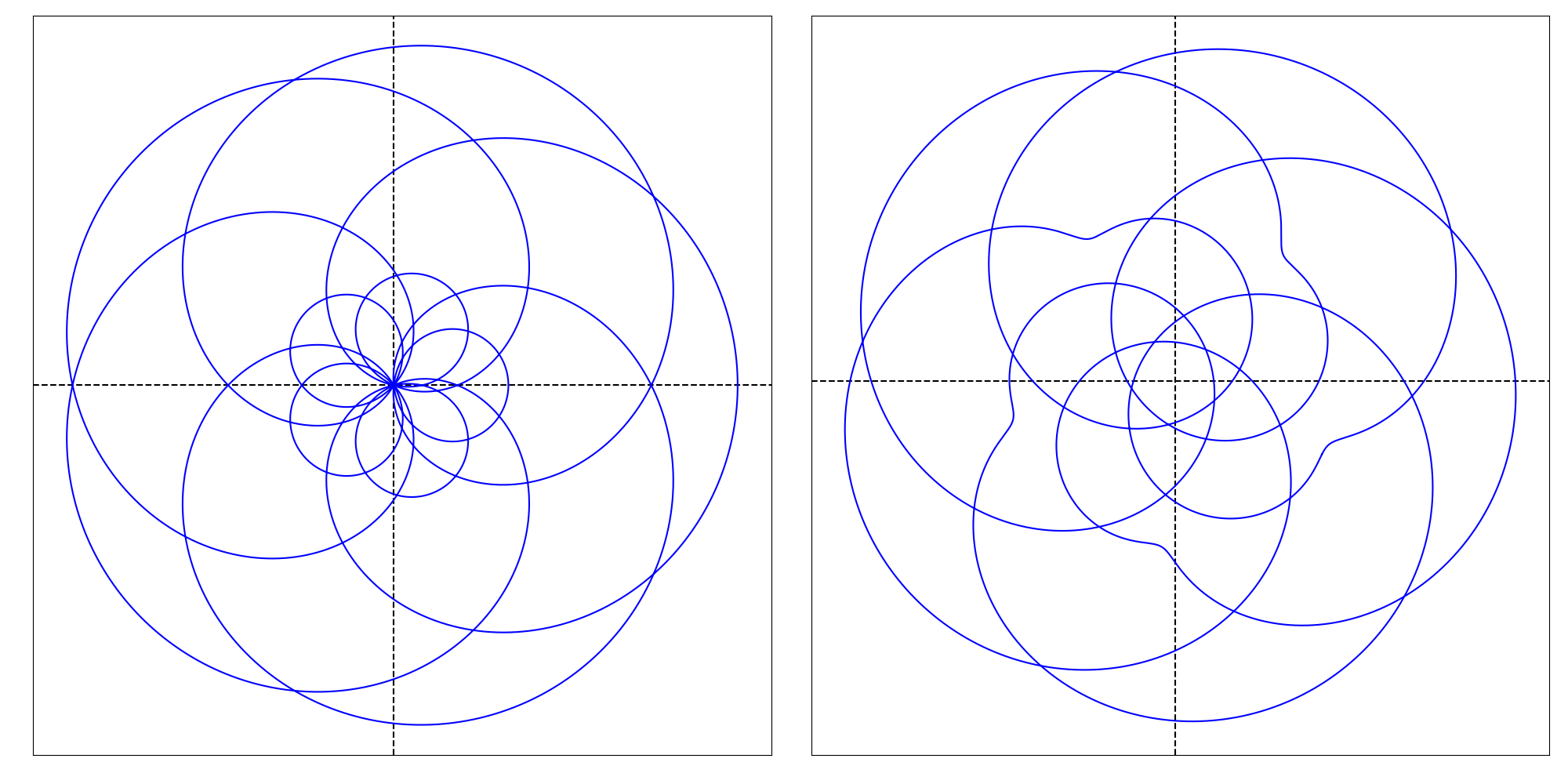}
    \caption{Two polynomials of symmetry of type $(2, 5).$ On the left,  $p(z) = z^2 + z^7 + z^{12}$ that has real coefficients. The symmetry group of $p(S^1)$ contains both a rotational symmetry of order $5$ and a mirror symmetry along the $x$-axis. On the right, the polynomial $p(z)=2 z^2-2iz^7+iz^{12}$ has no mirror symmetries, i.e. the image is chiral, hence the symmetry group of $p(S^1)$ contains no elements of order two}
    \label{fig:symmetry25}
  \end{center}
  \end{figure}
\begin{example}
    If the exponents of the polynomial $p$ are such that $a_1$ is not coprime with $m=\GCD(a_k-a_j)$ then the symmetry group of $p$ need not contain the cyclic group of order $m.$ Consider $p(z) = z^5 + z^{10} + z^{15}$ where we have that $a_1 = 5$ is not coprime with $m=5.$ We note that $p(\e{ 2 \pi i/5} \cdot z) = p(z)$ but $p(S^1)$ is not invariant under the rotation of angle $2 \pi/5.$
\end{example}
\subsection{Permutations on zeros and poles}
Let us slightly generalise our setting and take $f$ to be a meromorphic function defined on (a domain of) $\CC$ containing $S^1$ and let us study the relations between transformations of $f(S^1)$ and the zeros and the poles of $f.$
\begin{proposition} \label{prop:permutation}
    Let $g$ be a symmetry of $f(S^1)$ where $f(S^1)$ is not a reparametrization of the unit circle. Then $g$ induces a permutation on the zeros of $f$ and the poles of $f.$
\end{proposition}
\begin{proof}
 Let $z_0 \in \CC$ be a zero (resp. a pole) of $f$ and let $g$ be a planar rotation in the symmetry group of $f(S^1)$, that by Theorem~\ref{thm:symmetry} can be an angle $2\pi k/m.$ Then we have that 
 \[
    f(g z_0)=f(\e{ 2 \pi i/m} z_0) = \e{ 2 \pi i k/ m} f(z_0) = 0.
    \]
 If $g$ is of the form $g(z) = \e{ 2 i \sigma} \overline z$ then $f(gz_0) = f(\e{ 2 \sigma i} \overline{z_0}) = \e{2 i \sigma} \overline{f(z_0)} = 0.$
If $z_0$ is a pole of $f$, apply the same argument on $1/f$ to reach the same conclusion.
\end{proof}

The examples in \cite{poelke2014complex} suggest using Proposition~\ref{prop:permutation} to construct curves with trivial isometry group. We demonstrate this in the following.

\begin{example}
Consider the set $Z = \{ 0, \frac 1 2, i\} \subset \CC$ that admits no planar nontrivial isometry carrying $Z$ into $Z.$ Consider a polynomial $p$ admitting such points as roots.
Then the curve $p(S^1)$ must have a trivial isometry group by Proposition~\ref{prop:permutation} and is shown in Figure~\ref{fig:asymmetric}.
\end{example}
\begin{figure}
\begin{center}   
    \includegraphics[scale=0.5]{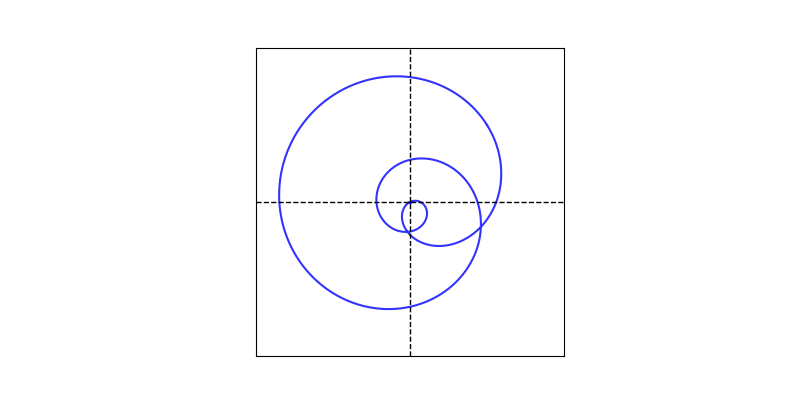}
    \caption{Image of the unit circle under the polynomial $p(z) = z \bigl (z- \frac 1 2 \bigr )(z-i)$}
    \label{fig:asymmetric}
  \end{center}
\end{figure}

\subsection{Algebraic varieties}
It has been proved by Quine~\cite{quine1976geometry} that the image of the circle under a complex polynomial $p$ of degree $n$ lies in a real algebraic variety $V \subset \CC$ of degree $2n.$ Generalizations of this result to Laurent polynomials have been proved later in \cite[Theorem 2.1]{kovalev2020algebraic}. 

Such a variety has by definition, a rational parametrization and it was shown that an implicit equation may be found by computing the resultant $h(w, \overline w)$ of the polynomials $A(z) = p(z)-w$ and $B(z) = z^n (\overline p(1/z)-\overline w)$, in which $w$ is a generic point $w \in p(S^1).$
Recall, that the resultant of two univariate polynomials is another polynomial which is defined as the determinant of the Sylvester matrix of the two polynomials.
If $\alpha_1, \ldots, \alpha_n$ are the complex roots of $A$ (depending on $w$), then $1/\alpha_1, \ldots, 1/\alpha_n$ are the roots of $B$ and by standard results (see e.g. the classical~\cite{gelfand2009discriminants}) their resultant is
\begin{equation}\label{eq:resultant}
h(w, \overline w) = (-1)^{n^2} \prod_{k=1}^n A \biggl (\frac 1 {\alpha_k} \biggr )
\end{equation}
where we consider conjugate coordinates $(w, \overline w)$ of $\CC.$ If $\sigma$ is a linear transformation of $\CC$, then its action in such coordinates can be written as $\sigma \cdot (w, \overline w) = (\sigma w, \sigma^{-1} \overline w).$

If $p$ admits a symmetry $\sigma$, we want to infer that the variety $V$ is also invariant under $\sigma$ or an integer power of $\sigma.$ This is done in the following Proposition.
\begin{proposition}
    Let $p$ be a complex Laurent polynomial admitting a symmetry $\sigma$ of type $(k, m)$ and let $p(S^1)$ sit in the real variety $V = \{ h=0 \} \subset \CC.$ Then $V$ is $\sigma^k$-invariant. Moreover, $V$ is invariant under mirror symmetries.
\end{proposition}
\begin{proof}
    We want to prove that the polynomial defining $V$ is $\sigma^k$-invariant. Note that $\sigma \alpha_\ell$ is a root of $p(z) - \sigma^k w.$ Using~\eqref{eq:resultant} and Definition~\ref{def:symm_laurent} we compute
    \begin{align*}
        h(\sigma^k w, \sigma^{-k} \overline w) &= (-1)^{n^2} \prod_{\ell=1}^n \biggl (p \biggl ( \frac 1 {\sigma \alpha_\ell} \biggr ) - \sigma^{k} w  \biggr ) \\
                                             &= (-1)^{n^2} \prod_{\ell=1}^n \biggl (\sigma^k p \biggl ( \frac 1 {\alpha_\ell} \biggr ) - \sigma^{k} w  \biggr ) \\
                                             &= (-1)^{n^2} \prod_{\ell=1}^n \sigma^k \biggl ( p \biggl ( \frac 1 {\alpha_\ell} \biggr ) - w  \biggr ) \\
                                             &= (-1)^{n^2} \prod_{\ell=1}^n \sigma^k A \biggl ( \frac 1 {\alpha_\ell} \biggr ) \\
                                             &= \sigma^{kn} h(w, \overline w)
                                            \end{align*}
    So if $w \in V$, then it satisfies $h(w, \overline w) = 0 = h(\sigma^k w, \sigma^{-k} w)$ and then $\sigma^k w \in V.$

    If $\sigma \colon z \mapsto \e{2 \pi i \tau} \overline z$ is a mirror symmetry, we note that $\sigma \alpha_\ell$ is a root of $p(z)-\sigma w$ since
    \begin{align*}
        p(\sigma \alpha_\ell) - \sigma w &= \e{ 2 \pi i \tau} \overline{p(\alpha_\ell)} - \sigma w \\
                                         &= \e{ 2 \pi i \tau} \biggl ( \overline{p(\alpha_\ell) - w} \biggr ) \\
                                         &= \sigma \overline{A(\alpha_\ell)}\\
                                         &= 0.
    \end{align*}

    So $1/\sigma \alpha_\ell$ is a root of $p(1/z)-\sigma w$ and we can compute
    \begin{align*}
        h(\sigma w, \overline{\sigma w})    &= (-1)^{n^2} \prod_{\ell=1}^n \biggl ( p(\sigma \alpha_\ell^{-1}) - \sigma w \biggr ) \\
                                            &= (-1)^{n^2} \prod_{\ell=1}^n \e{ 2 \pi i \tau} \biggl ( \overline{p(\alpha_\ell^{-1})} - \sigma w \biggr ) \\
                                            &= (-1)^{n^2} \prod_{\ell=1}^n \sigma A \biggl ( \frac 1 {\alpha_\ell} \biggr ) \\
                                            &= \sigma^n h(w, \overline w)
    \end{align*}
    which is the same as $h(w, \overline w)$ if $n$ is even or $\sigma h(w, \overline w)$ if $n$ is odd.
\end{proof}
\subsection{Irrational case}
Let $a_1 < a_2 < \ldots < a_n$ be a sequence of real numbers and let $w_k \in \CC$ be arbitrary. Consider the curve 
\[
\gamma(t) = \sum_{j=1}^n w_j \e {2 \pi i a_j t} \textup{ for $t \in (-\infty, \infty)$}
\]
and let us consider the symmetry group of $\gamma(\RR).$ 

The necessity of considering the image of the whole real line is suggested by the following fact.
\begin{proposition}
    If the curve $\gamma \colon \RR \to \CC$ is periodic, then the $a_k$ span a one-dimensional vector space over $\QQ.$
\end{proposition}
\begin{proof}
    Let $T$ be the period of $\gamma.$ Then by definition it holds that 
    \[
       \gamma(t) = \sum_k w_k \e{ 2 \pi i a_k t} = \gamma(t+T) = \sum_k w_k \e{ 2 \pi i a_k (T+t)}\]
        from which we see that every exponential term must equal its counterpart by linear independence. Hence, there exist integers $\lambda_k$ such that $a_k T = \lambda_k.$ This means that every $a_k$ is a rational multiple of $1/T.$
\end{proof}
Concerning the symmetry group of $\gamma(\RR)$, we have the following.
\begin{theorem} \label{thm:sym1}
    Let the exponents $a_k$ be real and rationally independent, i.e. linearly independent over $\QQ.$ Then the isometry group of $\gamma(\RR)$ is a group of order 2 consisting only of the conjugation map $z \mapsto \overline z.$
\end{theorem}
\begin{proof}
    The fact that the complex conjugation is always a transformation of the curve image of the whole line is due to the fact that $\overline{\gamma(t)} = \gamma(-t).$
   Let $\sigma$ be a planar rotation given by multiplication with the complex number we still call $\sigma = \e{ 2 \pi i \alpha}.$
    Then one needs to have
    \[
        \sigma \gamma(0) = \e{2 \pi i \alpha} \sum_k w_k = \sum_k w_k \e{2 \pi i a_k t' }\]
        for some real $t'.$ We must then have that $a_k t' - \alpha$ is an integer.
        Assume for a contradiction that $t' \neq 0.$ Then by subtracting such relations for all $k, \ell$ with $a_\ell t' - \alpha \in \ZZ$, we get that there exist integers $\lambda_{k \ell}$ such that
        \[
        (a_k - a_\ell) t' = \lambda_{k \ell} \in \ZZ
        \]
        
        If $t' \neq 0$ we have that $\lambda_{k \ell} \neq 0$ and can then write
        \[
            \frac{a_i - a_j}{\lambda_{ij}} = \frac{a_k - a_\ell}{\lambda_{k \ell}}
            \]
        that would mean there exist a rational linear combination of the $a_k$'s giving zero. This contradicts the assumption that the $a_k$ are independent. Hence, $t' = 0$ which means that $\alpha \in \ZZ$ and $\sigma = 1.$
\end{proof}

\begin{theorem}\label{thm:dense_annulus}
    If the $a_k$ and 1 are rationally independent, then the image $\gamma(\RR)$ is dense in an annulus of the complex plane. 
\end{theorem}
For this, we need a lemma.
\begin{lemma} \label{lemma:Pi}
    Let $T^n$ be the $n$-dimensional torus and consider the map $\Pi\colon T^n \to A$ defined by $\Pi(x_1, \ldots, x_n) = \sum_k w_k x_k$ and where $w_k \in \CC$ and $A$ is the annulus 
    \[
    A = \{ z \in \CC: \min_{\RR} |\gamma(t)| \leq |z| \leq \max_{\RR} |\gamma(t)| \}.     
    \]

Then $\Pi$ is continuous and surjective.
\end{lemma}

\begin{proof}
    Let $z = \rho \e{ 2 \pi i \phi} \in A.$ By construction, there exist a $t \in \RR$ such that $|\gamma(t)| = \rho$ and such $t$ is hence mapped by $\gamma$ onto a point on the circle of radius $\rho$ containing $z.$
    So we can write, for some angle $\sigma$, that $z = \e{2 \pi i \sigma} \gamma(t).$
    If we take $x_k = \e{2 \pi i(a_k t + \sigma)} \in S^1$ we obtain that
    \[
    \Pi(x_1, \ldots, x_n) = \sum_k w_k \e{ 2 \pi i(a_k t + \sigma)} = \e{2 \pi i \sigma} \gamma(t) = z.
    \]
\end{proof}
\begin{proof}[Proof of Theorem~\ref{thm:dense_annulus}]
    Consider the map $f\colon \RR \to T^n$ given by 
    \[ 
        f(t) = (\e{2 \pi i a_1 t}, \ldots, \e{ 2 \pi i a_n t} ). \]
         By Kronecker's Theorem, the image $f(\RR)$ is dense in the $n$-dimensional torus if the $a_k$ are rationally independent.

    Compose $f$ with the map $\Pi$ of Lemma~\ref{lemma:Pi} to obtain $\Pi \circ f\colon \RR \to A.$
    Since $\Pi$ is continuous and surjective, the closure of $\Pi(f(\RR))$ in $A$ coincides with $\Pi(\overline{f(\RR)}) = \Pi(T^n) = A$, where $\overline{f(\RR)} = T^n$ denotes the closure in $T^n$ of the dense subset $f(\RR) \subset T^n.$ This shows that $\gamma(\RR) = \Pi(T^n)$ is dense in the annulus $A.$
\end{proof}
\begin{remark}
The problem of estimating the thickness of such annulus boils down to the minimization of $\|\gamma\|$ over $\RR$, since the maximum over $\RR$ of $\| \gamma \|$ is always attained at $t=0$ and is equal to the modulus of the sum of the polynomial coefficients.
In the binomial case with coefficients $v, w$, the square distance from the origin can be estimated as 
\[
\| \gamma(t) \|^2 = |v|^2 + |w|^2 -2 \Re (v \overline{w}  e^{2 \pi i (b-a)t})  \geq |v|^2 + |w|^2 - 2 |v||w| = |v-w|^2 \geq 0,
\]
so when the two weights are different the annulus is not a disk. Note that this is independent from the exponents $a, b \in \RR.$ The authors do not know any estimates for the case of more than two terms.

We also note that the annulus cannot degenerate to a circle. Indeed, from the differential geometric viewpoint the curve defined by a Laurent polynomial $f$ coincides with $t \mapsto C \exp (2 \pi i m t)$ for some $m \in \ZZ.$
Up to multiplying by a power of $z$, we can assume that $f$ is a complex polynomial that, by a maximum argument principle, must be of the form $a z^n$ for some $n \in \ZZ$ and some $a \in \CC$ of modulus $c.$
\end{remark}

%%%%%%%%%%%%%%%%%%%%%%%%%%%%%%
%%%%%%%%%%%%%%%%%%%%%%%%%%%%%%
\section{Evolution under the wave equation} \label{sec:wave}
Following the ideas of Farris~\cite{farris2015creating} one can consider the initial curve 
$$\gamma = \gamma(0)\colon x \mapsto p(\e{2 \pi i x})$$ 
for some Laurent polynomial $p$, evolving under the wave equation, that is to consider a solution of the one-dimensional wave problem
\begin{equation} \label{eq:wave}
    \begin{array}{rl}
    u_{tt} &= c^2 u_{xx}  \\ 
    u(x, 0) &= \gamma(x)\\
    u_t(x, 0) &= f(x)
\end{array}
\end{equation}
for smooth $u \colon [0, 1] \times \RR \to \RR$, some initial condition $f \colon [0, 1] \to \RR$ and a \emph{speed} parameter $c \in \RR.$

It is known that all solutions are given by the D'Alembert formula. We can derive the following particular solution either from D'Alembert's formula or via a simple \emph{guess and check} ansatz.
Setting $f=0$, we find that the function
\begin{equation}
    \label{eq:wave_poly}
    u(x, t) = \sum_n c_n \cos(2 \pi n c t) \e{ 2 \pi i n x},
\end{equation}
is a solution of the problem; i.e. the expansion of $\frac 1 2 (\gamma(x-ct)+\gamma(x+ct))$ using the trigonometric definition of $\gamma.$
We can prove some properties of the function $u$ and the family of curves it defines.
\begin{theorem} \label{thm:waveprop}
    \begin{enumerate}[(i)] Let $\gamma(t) = u (\cdot, t)$ be the family of curves defined by $u$ in~\eqref{eq:wave_poly}.
        \item \label{wave_prop} The family of curves $\gamma(t)$ satisfies the wave problem~\eqref{eq:wave} with initial condition $f=0.$
        \item \label{period_wave} The period with respect to $t$ of $\gamma(t)$ is 
        \[ T = \frac 1 {c\GCD \{n:c_n \neq 0\}},
        \] that is $\gamma(t) = \gamma(t+T)$ as curves $[0, 1] \to \CC.$
        \item \label{symmetry_wave} The symmetry group of the curve $\gamma(t)$ contains the symmetry group of $\gamma = \gamma(0).$
    \end{enumerate}
\end{theorem}
\begin{proof}
    \eqref{wave_prop} Solving the one-dimensional wave problem by series yields the expression~\eqref{eq:wave_poly} solving~\eqref{eq:wave} with initial condition $u_t(x, 0)=0.$
    
    To prove~\eqref{period_wave}, consider the family of linearly independent functions $x \mapsto \e{2 \pi i n x}$ on the unit circle. By applying the periodicity condition $\gamma(t+T) = \gamma(t)$ to each component of the expansion~\eqref{eq:wave_poly} we obtain the system
    \[
        c_n (\cos(2 \pi c n t) - \cos (2 \pi c n (t+T))) = 0 \]
        holding for all $n$ and all $t \in \RR.$ This yields that, for all $n$, either $c_n = 0$ or that $cnT$ is an integer. It follows that the smallest allowable $T$ that works for all $n$ such that $c_n \neq 0$ is $T = 1 /(c\GCD \{n:c_n \neq 0\}).$
    
        We have remarked after Theorem~\ref{thm:symmetry} that a change of the polynomial coefficients does not affect the existence of rotational symmetries. A mirror symmetry of $\gamma(0)$ would still hold for $\gamma(t)$ because of the coefficient perturbation by the real $\cos 2 \pi c n t \in \RR.$
        Hence, every element of such a group also defines a symmetry of $\gamma(t)$ and this proves~\eqref{symmetry_wave}.
\end{proof}
\begin{remark}
    We note that the symmetry group of $\gamma(t)$ can ``explode'' to a bigger finite group or even  ``degenerate'' to the continuous circle group for special values of $t$ that make coefficients of $u(x, t)$ vanish, i.e. zeros of $\cos (2 \pi cnt).$ We shall see some examples of the latter case in our closer look at the sum of two exponentials.
\end{remark}
%\subsection{Evolution of the sum of two exponentials}
In the remainder of this section, we consider the special case of the sum of two exponentials as it lets us carry out computations in detail. Consider the complex Laurent polynomial 
\begin{equation}
    \label{eq:poly2terms}
    p(z) = c_a z^a +c_b z^b
\end{equation} for integer exponents $a<b$ and complex coefficients $c_a$ and $c_b.$
Then the one-dimensional solution of the wave problem~\eqref{eq:wave} with $f=0$ is the special case of~\eqref{eq:wave_poly}
\begin{equation*}
    \label{eq:wave2terms}
    u(x, t) = c_a \cos (2 \pi a c t )\e{2 \pi i a x} + c_b \cos (2 \pi b c t) \e{ 2 \pi i b x}.
\end{equation*}

\subsection{Sum of two exponentials - differential geometric properties and winding number}
Applying Theorem~\ref{thm:waveprop} we can consider $\gamma(t)$ for $t \in [0, T]$, where $T = 1/(c\GCD(a, b)).$ Moreover, every transformation in the symmetry group $D_{b-a}$ is also a transformation of $\gamma(t).$ However, for $t$ such that $\cos (2 \pi a ct) = 0$ or $\cos (2 \pi b c t )= 0$, the curve $\gamma(t)$ ``degenerates'' to a circle, making the symmetry group ``explode'' to the non-discrete $S^1.$

In analogy with~\cite[Theorem 1]{pausinger2023symmetry}, it is possible to prove a theorem about the winding number of $\gamma(t).$
The main idea is that if the coefficient of $z^a$ dominates, the winding number will be $a$ and if the coefficient of $z^b$ dominates, it will be $b.$ To formalize this intuition, and also detect more intricate behaviour, we introduce several functions based on Chebyshev polynomials in the following three lemmas. The functions are then utilized in Theorem 5 to study the winding numbers of $\gamma(t).$ We refer to Example \ref{ex:evolution_2z5z2}, and to Figures \ref{fig:plotphiz22z5} and \ref{fig:plotevolz22z5} for an illustration of the result.

First, recall that it is possible to express $\cos (dx)$ for integer $d$ in terms of $y=\cos( x)$ through the $d$-th Chebyshev polynomial $T_d$, i.e. for all $x \in [-\pi, \pi]$ we have that $$T_d(\cos (x)) = \cos (dx).$$ Such polynomials satisfy the recurrence relation $T_0(y)=1, T_1(y)=y$ and $T_{d+1}(x) = 2x T_d(x) - T_{d-1}(x).$
The Chebyshev polynomial of second kind $U_k(y)$ expresses $\sin (2 \pi \theta) \cdot \sin (2 \pi k \theta)$ in terms of $y=\cos (2 \pi \theta).$ The range of $U_n$ is $[-(n+1), n+1].$
For an introduction to the topic, we refer to~\cite[Chap.~22]{abramowitz1965handbook}.
\begin{lemma}\label{lemma:cheby}
    For odd $d$, the Chebyshev polynomials satisfy
    \begin{equation*}
        \label{eq:cheby}
        |T_{d}(y)| \leq d |y|
    \end{equation*}
    for all $y \in [-1, 1]$
\end{lemma}
\begin{proof}
Let $d=2n+1$ for an integer $n$ and observe that $T_{2n+1}'(0)=2n+1.$
Moreover, note that $T_{2n+1}(0) = 0$ and that the statement can be rewritten as $|T_{2n+1}(y) - T_{2n+1}(0)| \leq |T'_{2n+1}(0)||y-0|$ since $T_{2n+1}$ is a Lipschitz continuous function.
\end{proof}
For the polynomials of the second kind we have the following.
\begin{lemma}
    \label{lemma:psi}
    Let $I:= [0, 1] \setminus \{t: \sin (2 \pi b t) = 0 \}$ and define $\psi\colon I \to \RR$ be the function defined as $\psi(\theta) = \frac{\sin (2 \pi a \theta)}{\sin (2 \pi b \theta)}.$ If $a|b$, extend $\psi$ to the points $\theta_k$ of the form $\ZZ/a$ where $\sin (2 \pi a \theta_k) = \sin (2 \pi b  \theta_k) = 0$ to assume value $a/b.$ Then such extension is continuous and the range of $|\psi|$ is
    \begin{equation*}
        |\psi|(I) =
        \begin{cases}
            [0, +\infty) \textup{ if $a \nmid b$} \\
            \bigl [\frac  a b, +\infty \bigr ) \textup{ if $a | b$}.
        \end{cases}
    \end{equation*}
    In the latter case, the minimum $\frac a b$ is attained exactly at the points $\theta_k.$
\end{lemma}

\begin{proof}
If $a \nmid b$ then there exists $\theta_0 \in [0,1]$ such that $\sin (2 \pi b \theta_0 )= 0$ and $\sin (2 \pi a \theta_0) \neq 0$, so we have that the range of $|\psi|$ is $[0, +\infty).$ 
    If $b=da$ then we can write for $s = \cos (2 \pi \theta)$ and $y=\cos (2 \pi a \theta)$
    \[
    \psi(\theta) = \frac{ U_{a-1}(s)}{ U_{da-1}(s)} = \frac{U_{a-1}(s)}{U_{d-1}(T_a(s)) U_{a-1}(s)} = \frac 1 {U_{d-1}(y)}
    \]
    Since it is known that the range of $U_{d-1}$ is $[-d, d]$ we conclude.
    For $\theta_k$ as in the statement in the case $b=da$ let us compute
    \begin{equation*}
        \lim_{\theta \to \theta_k} \frac{\sin (2 \pi a \theta)}{\sin (2 \pi b \theta)}  = \lim_{\theta \to \theta_k} \frac{a \cos (2 \pi a \theta)}{b \cos (2 \pi b \theta)}= \frac a b \frac{1}{(-1)^d},
    \end{equation*}
    so we obtain the second statement.
\end{proof} 
An analog holds for the function defined by the ratio of cosines.
\begin{lemma}
    \label{lemma:phi}
    Let $I_T:= [0, T] \setminus \{t: \cos 2 \pi a c t = 0 \}.$ 
    For fixed $c>0$, consider the function $\phi \colon I_T \to \RR $ defined by $\phi(t) = \frac{\cos (2 \pi bct)}{\cos (2 \pi a ct)}.$ In case $a|b$ and $2a\nmid b$, extend it to the points of the form 
    $t_k = \frac 1 {4 a c} + \frac k {2 a c}$ where $\cos (2 \pi a c t )= \cos (2 \pi b c t )= 0$ to assume value $b/a.$
    Then the extension is continuous and the range of $|\phi|$ is
    \begin{equation*}
        |\phi|(I_T) =
        \begin{cases}
            [0, +\infty) \textup{ if $a \nmid b$ or $2a | b$} \\
            \bigl [0, \frac b a \bigr ] \textup{ if $a | b$ and $2a \nmid b$}.
        \end{cases}
    \end{equation*}
    In the latter case, the maximum $\frac b a$ is attained exactly at the points $t_k.$ 
\end{lemma}
\begin{proof}
    If $a \nmid b$, then there exists $t_0 \in [0, T]$ such that $\cos (2 \pi a ct_0) = 0$ and  $\cos (2 \pi b c t_0) \neq 0.$ Then 
    \begin{equation*} %\label{eq:limitzero}
        \lim_{t \to t_0} |\phi(t)| = + \infty 
    \end{equation*}
and since $[0, T]$ contains also points that make $\cos (2 \pi bct)$ vanish, we get that the range of $|\phi|$ is $[0, \infty).$

    If $a | b$ we can write $b = da$ and express $\phi(t)$ as the ratio of the Chebyshev polynomials $T_d (y)/y$ for $y = \cos (2 \pi a c t).$ We have that $T_d(0) = 0$ if $d$ is odd.
    So if $2a | b$ we can choose $d$ to be even and get that $\lim_{y \to 0} |T_d(y)/y| = + \infty$ so that the range of $|\phi|$ is $[0, +\infty).$
    If $2a \nmid b$, then $d$ is odd, and we apply Lemma~\ref{lemma:cheby} to obtain that $|\phi| \leq d = \frac b a.$

    For the last assertion, compute the limit 
    \[
        \lim_{t \to t_k} \phi(t) = \lim_{t \to t_k} \frac{ b \sin (2 \pi b c t)}{a \sin (2 \pi a c t)} = \frac b a. \qedhere 
        \]
\end{proof}
\begin{remark}
    Note that although its period may be different from $T$, for the study of the evolution we are only interested in the behavior of $\phi$ for $t \leq T.$
\end{remark}
We are now ready to prove our winding number theorem. Recall that by the classical argument principle, if $0$ is not in its support, the winding number of the curve $f(S^1)$ with respect to $0$ is the number of zeros of $f$ inside the unit disk minus the number of poles of $f$, (see e.g.~\cite{ahlfors1979complex}). This theorem helps to illustrate how the characteristics of the curve can change along its evolution under the wave equation.
\begin{theorem}
    \label{thm:winding2terms}
    Let $\gamma(t) = u(\cdot, t)$ be the family of curves defined by~\eqref{eq:wave_poly} for $t \in [0, T] \setminus \{t: \cos 2 \pi a c t = 0 \}$ and consider the function $\phi(t) = \frac{\cos( 2 \pi bct)}{\cos( 2 \pi a ct)}.$ Let $N(t, 0)$ be the winding number of $\gamma(t)$ with respect to zero. Then
    \begin{equation*}
        N(t, 0) = 
        \begin{cases}
            a \textup{ if $\cos (2 \pi b c t )= 0$ or $|\phi(t)|> \frac{c_a}{c_b}$ and $a > 0$} \\
            -a \textup{ if $\cos (2 \pi b c t )= 0$ or $|\phi(t)|> \frac{c_a}{c_b}$ and $a < 0$} \\
            \textup{undefined if $|\phi(t)|= \frac{c_a}{c_b}$}\\
            b \textup{ if $\cos (2 \pi a c t) = 0$ or $|\phi(t)|< \frac{c_a}{c_b}$ and $a>0$}\\
            b-2a\textup{ if $\cos (2 \pi a c t) = 0$ or $|\phi(t)|< \frac{c_a}{c_b}$ and $a<0.$}
        \end{cases}
    \end{equation*}
\end{theorem}
\begin{proof}
    Write the family of curves as
    \[
        u(x, t) = c_a \cos (2 \pi a c t) z^a \biggl ( 1+ \frac{c_b}{c_a} \phi(t) z^{b-a} \biggr ) \biggr |_{z=\e{2 \pi i x}}.
    \]
        With the idea of applying the argument principle, we want to count the zeros and the poles of a complex polynomial of the form $z^a(1+Kz^{b-a})$ as the real number $K = \frac{c_b}{c_a} \phi(t)$ varies. The possibilities can be summarized as follows:
        \begin{itemize}
            \item if $a>0$ and $|K|<1$ then the unit disk contains only the root $0$ with multiplicity $a;$
            \item if $a<0$ and $|K|<1$ then the unit disk contains only the pole $0$ with multiplicity $a;$
            \item if $a>0$ and $|K|>1$ then the unit disk contains the root $0$ with multiplicity $a$ and the root $-\frac 1 K$ with multiplicity $b-a;$
            \item if $a<0$ and $|K|>1$ then the unit disk contains the pole $0$ with multiplicity $a$ and the root $-\frac 1 K$ with multiplicity $b-a.$
        \end{itemize}
        This translates as
        \begin{align*}
        \textup{$\#$ of zeros} &- \textup{$\#$ of poles inside the unit disk} \\ 
        & =\begin{cases}
            a \textup{ if $|\phi(t)|> \frac{c_a}{c_b}$ and $a>0$}\\
            -a \textup{ if $|\phi(t)|> \frac{c_a}{c_b}$ and $a<0$}\\
            b \textup{ if $|\phi(t)|< \frac{c_a}{c_b}$ and $a>0$}\\
            b-2a \textup{ if $|\phi(t)|< \frac{c_a}{c_b}$ and $a<0$}
        \end{cases}    
        \end{align*}
        and that for $|\phi(t)|= \frac{c_a}{c_b}$ the curve crosses the origin, so the winding number is undefined.
\end{proof}

\subsection{Sum of two exponentials - singularities}
The singularities in the differential geometric sense of the curve $\gamma(t)$ are attained at points $x \in [0, 1]$ such that $u_x(x, t)=0.$
In general, we have
\[
u_x(x, t) = 2 \pi i \sum_n n \cos (2 \pi n c t) \e{2 \pi in x}
\]
In our particular case we have
\begin{eqnarray}
\label{eq:deriv2terms}
u_x(x, t) &= 2 \pi i (a \cos (2 \pi a c t) \e{2 \pi i a x} + b \cos (\pi b c t) \e{ 2 \pi i b x} ) \\
        &= 2 \pi i a \cos (2 \pi a c t )\cdot z^a (1 + \frac b a \frac{c_b}{c_a} \phi(t) z^{b-a})|_{z=\e{2 \pi i  x} } 
\end{eqnarray}
The study of the existence of zeroes on the unit circle again boils down to the study of the function $\phi$ from Lemma~\ref{lemma:phi}, so we obtain the following facts.
\begin{proposition}
The curve $\gamma(t)$ admits singular points $(x_0, t_0)$ for $t_0$ such that 
$$|\phi(t_0)| = \frac{|c_a|}{|c_b|} \frac a b$$ 
and $x^k$ with $x = \frac 1 {2 \pi (b-a)}$ for $k=0, \ldots, b-a-1$, i.e., at the $b-a$-th roots of unity. Such points are cusps, i.e. $u_x(x_0, t_0)=0$ and $u_{xx}(x_0, t_0) \neq 0.$
\end{proposition}
\begin{proof}
    The form of the solutions $(x^k, t_0)$ follows directly from~\eqref{eq:deriv2terms}. Differentiate again and evaluate at such points to obtain
\[
    u_{xx}(x^k, t_0) = - 4 \pi^2 a \cos (2 \pi a c t ) x^{ka} \cdot 0 + 4 \pi^2 a \cos (2 \pi a c t )x^{ka}(b-a) \neq 0. \qedhere
\]
\end{proof}

\begin{example} \label{example:sing_future}
    Assume that $\gamma$ has at least two points $x_k \in [0, 1]$ such that $\gamma'(x_k)=0.$ Consider two such points $x_1 < x_2$ and the intersection $(x_0 = \frac{x_1+x_2}{2}, t_0)$ of the lines $x+ct = x_1$ and $x-ct=x_2.$
    Then at such point we have
    \[
    u_x(x_0, t_0) = \gamma'(x_0+ct_0) + \gamma'(x_0-ct) = \gamma'(x_1) + \gamma'(x_2) = 0,
    \]
    so we have constructed a singular point at a future time $t_0$ by \emph{following the characteristic lines}.
\end{example}

\begin{example}
    As in Example~\ref{example:sing_future} for $n=2$, we note that the presence of only one singularity at $t=0$ and $x=\frac 1 2$ is not enough to infer the existence of more singularities for $u(\cdot, t)$ for future $t>0.$
    Indeed, we have $u(x, t) =  \cos (4 \pi c t) \e{2 \pi i x} + \cos (2 \pi c t  )\e{ 4 \pi i x}$ and so
    \[
    u_x(x, t) = 2 \pi i \e{2 \pi i x} (\cos (4 \pi c t) + 2 \cos (2 \pi c t )\e{2 \pi i x})    
    \]
    that admits zeroes $(x, t)$ only for $t=0$ for which we recover the singularity $x=1/2.$
\end{example}

%%%%%%%%%%%%%%%%%%%%%%%%%%%%%%%%%%%%%%%%%
%%%%%%%%%%%%%%%%%%%%%%%%%%%%%%%%%%%%%%%%%
%%%%%%%%%%%%%%%%%%%%%%%%%
\section{Self-intersections} \label{sec:self}
Quine \cite{quine1973self} gives a characterization of the pairs $(z_1, z_2) \in S^1 \times S^1$ for which $p(z_1) = p(z_2)$ and finds the upper bound $2(n-1)^2$ for the number of such self-intersections of the image $p(S^1)$ where $n$ is the degree of $p.$

Furthermore, he shows that this estimate is sharp using the polynomial $z^n + \varepsilon z$ for small enough $\epsilon.$
Quine considers the open subsets of $S^1 \times S^1$ given by 
\begin{align*}
    A &=  (-1,1) \times S^1\\
    B &= \{(z_1, z_2) \in S^1 \times S^1: z_1 \neq z_2 \}
\end{align*}
and the diffeomorphism $\psi \colon A \to B$ given by $\psi(s, x) = (e^{i \theta}x, e^{-i\theta}x)$, where $s=\cos( 2 \pi \theta) .$

Quine's method to find points in $B$ that are mapped to the same value uses the properties of the \emph{Dieudonn\'e polynomial} that can be defined on $B$ or on $A$ as
\begin{align*}
    G(z_1, z_2) &= \frac{p(z_1) - p(z_2)}{z_1-z_2} \\
    &= \frac{p(e^{2 \pi i\theta}x)-p(e^{-2 \pi i \theta}x)}{x(e^{2 \pi i\theta} - e^{-2 \pi i\theta})} \\
    &= \sum_{k=0}^n c_{k+1} U_k(s) x^k \\
    &=: g(s, x)
\end{align*}
where $U_k$ is the $k$-th Chebyshev polynomial of second kind, that expresses $$\sin (2 \pi (k+1)\theta)/\sin (2 \pi \theta)$$ in terms of $s=\cos (2 \pi \theta).$

A zero $(s, x) \in A$ for $G$ is also a zero of $G(s, 1/\bar x)$, hence of the polynomial $G^*(s, x) = x^n G(s, 1/\bar x)$ and so of the resultant $R(t).$

We state a lemma concerning a property of such intersection points. Analogously to $g$, let us define the polynomial $h$ as 
\begin{equation*} %\label{eq:defh}
    h(s, x) = \sum_{k=0}^n c_k T_k(s) x^k
\end{equation*}
\begin{lemma} \label{lemma:parallel}
    For all $(s, x) \in A$ and its corresponding $\theta$ with $s = \cos (2 \pi \theta)$ we have
    \begin{equation*} %\label{eq:p_on_A}
        p(e^{2 \pi i\theta}x) = h(s, x) + i \sin (\theta) \cdot xg(s,x).
    \end{equation*}
    In particular, if $(s, x) \in A$ is a zero of $g$, then the point $p(e^{i \theta}x) \in \CC$ lies along the direction of ${h(s, x)}.$ 
\end{lemma}
\begin{proof}
    Using the definition of Chebyshev polynomials we can expand 
    $$e^{2 \pi ik\theta} = T_k(s) + i \sin (\theta) U_{k-1}(s)$$
    for $k\geq 0$ with the convention $U_{-1}(s) = 0.$ We then compute 
    \begin{eqnarray*}
        p(e^{2 \pi i\theta}x) &= \sum_{k=0}^{n} c_k (T_k(s) + i \sin (\theta ) U_{k-1}(s))x^k \\
        &= h(s, x) + i \sin (\theta) \cdot x g(s, x).
    \end{eqnarray*}
    To prove the last statement, multiply by $\overline{h(s, x)}$ and note that
    \[    p(e^{2 \pi i\theta}x) \overline{h(s, x)} = | h(s, x)|^2 \in \RR
    \] 
    has zero imaginary part, meaning that the directions of $p(e^{i\theta}x)$ and ${h(s, x)}$ are parallel in the plane.
\end{proof}

\subsection{An application of Quine's method}
Let us now consider the simple case $p(z) = vz^a + wz^b$ and apply Quine's method to find the self-intersections. Then we have the following result, interesting on its own.
\begin{theorem} \label{thm:selfintersections}
    Let $\gamma$ be the image of the circle under $p(z) = vz^a + wz^b.$ Then
    \begin{enumerate}[(i)]
        \item The self-intersections of $\gamma$ lie on the directions of $\pm \eta^k$ with $\eta = \e{\frac{2 \pi i a}{b-a} }$ for $k = 0, \ldots, b-a-1;$
        \item On every such direction, the self-intersections are twice the number of solutions of the polynomial equation $U_{b-1}(s) =  \mp v/ (w U_{a-1}(s))$ and have modulus $|v T_a(s) \mp w T_b(s))|.$
    \end{enumerate}
\end{theorem}
\begin{proof}
    For our $p$ we have
    \begin{equation*}
        g(s, x) = v U_{a-1}(s) x^{a-1} + w U_{b-1}(s) x^{b-1}.
    \end{equation*}
    and 
    \begin{equation*}
        h(s, x) = v T_a(s) x^a + w T_b(s) x^b
    \end{equation*}
    If $U_{a-1}(s) = U_{b-1}(s) = 0$, then $s$ solves $\sin(2 \pi a \theta) = \sin(2 \pi b \theta) = 0$ and hence 
    \[
        \theta \in \frac{\ZZ}{b} \cap \frac{\ZZ}{a} = \frac{\ZZ}{\GCD(a, b)}
        \]
        hence $2 \theta = k/(2g) $ for some integer $0\leq k < 2g$ and $x \in S^1.$
     This case leads to no solutions, since the number of solutions of $g=0$ must be finite.
        If $U_{a-1}(s) \neq 0$, then we can rewrite our equation as 
        \begin{equation*}
             vx^{a-1} U_{a-1}(s) \left (1 + \frac w v \frac{U_{b-1}(s)}{U_{a-1}(s)} x^{b-a} \right ) = 0
        \end{equation*}
        and its solutions are in the set
        \begin{equation*}
            S_- \times R_+ \cup S_+ \times R_-
        \end{equation*}
        where
        \[
            S_\pm = \biggl \{ s \in (-1,1):  \frac{U_{b-1}(s)}{U_{a-1}(s)} = \pm \frac v w \biggr \}
            \]
            and 
            \[
                R_{\pm} = \{ x \in S^1: x^{b-a} = \pm 1 \}.
                \]
                For such pairs, we have that $h(s, x) = x^a (v T_a(s) \mp w T_b(s))$ and hence the direction of $p(e^{i \theta}x)$ is parallel to the one of $\bar x^a.$
                The modulus of $p(e^{i \theta}x)$ coincides in this cases with 
                \[
                    |h(s, x)| = |v T_a(s) \mp w T_b(s)| \]
                    proving the last statement.
                \end{proof}
                \begin{example}
                    Consider $p(z) = 2 z^3 + z^5$ and compute $$g(s, x) = 2U_2(s) x^2 + U_4(s) x^4$$ 
                    having roots $(s, x)$ with $x^2=\pm 1$ and $\theta$ satisfying
                    \[\pm 2 \sin (6 \pi \theta) = \sin (10 \pi \theta).
                    \]
                    Hence, the self-intersection points are to be expected on real and imaginary axes and their possible modulus is
                    \[
                        |h(s, \pm 1)| = |h(s, \pm i)| = | 2 T_3(s) + T_5(s)| 
                        \] 
                        Note that for $x=\pm i$ the possible $s$ are solutions of $$(4s^2-1)^2/(16s^4 - 12s^2 - 1)^2 = 4.$$
                    \end{example}
                    
                    We note that the solutions of $g$ of the form $(1, x)$ correspond to the points where the parametrization is not regular. Indeed, at such points one has
                    \[
                        \gamma'(t) = p'(e^{2 \pi i \theta}) = 0
                        \]
                        and $p'$ coincides with $g(1, \cdot)$ on $S^1.$
                        \begin{example}
                            Let us construct an example providing points of arbitrary multiplicity. Let $n$ be a positive integer and consider $p(z) = 1 + z + \ldots + z^n = \frac{z^{n+1}-z}{z-1}$ and the points $0, -1 \in \CC$ that have preimages
                            \[
                                p^{-1}(0) \cap S^1 = \biggl \{\e{2 \pi i \frac{k}{n}}: k=1, \ldots, n \biggr \} \]
                                and 
                                \[
                                    p^{-1}(1) \cap S^1 = \biggl \{ \e{2 \pi i \frac{k}{n+1}}: k=1, \ldots, n+1 \biggr \}
                                    \]
                                    of size $n$ and $n+1.$ This means that the curve crosses $n$ times the origin and $n+1$ times the point $1 \in \CC.$
                                Such points are ordinary since at all such points the derivative $p'$ assumes different values. The rosette is shown in Figure~\ref{fig:multiple}.
                        \end{example}
                        \begin{figure}
                        \begin{center}                           
                            \includegraphics[scale=0.5]{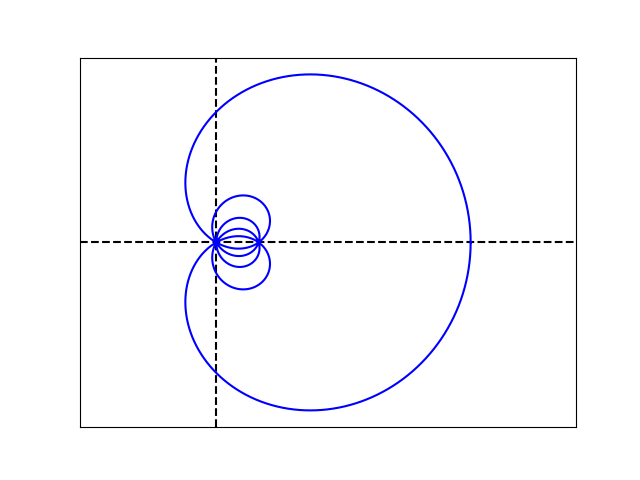}
                            \caption{The rosette for $p(z) =  1+ z + z^2 + z^3 + z^4 + z^5$, that crosses $5$ times the origin and $6$ times the point $(1, 0)$}
                            \label{fig:multiple}
                        \end{center}
                        \end{figure}                            
                                \begin{example} \label{example:cusp}
                                    The image of $p(z) = z^n+nz$ has a cusp. Differentiate to get $p'(z) = n(z^{n-1}+1)$, which has the $(n-1)$-th roots of $-1$ as roots, lying on the unit circle, and note that the derivative has all its $n-1$ zeros on the unit circle. Such points are cusps since they are not zeros of $p''(z) = n(n-1)z^{n-1}.$ We note this is the only way to obtain singularities \emph{of the parametrization} $t \mapsto p(\e{ 2 \pi i t}).$
                                \end{example}
\subsection{Self-intersections along the wave flow}
Let us now apply Theorem~\ref{thm:selfintersections} for $v = v(t)= c_a \cos (2 \pi c a t)$ and $w = w(t)= c_b \cos (2 \pi c b t)$ to study the time evolution of the rosette associated to the polynomial~\eqref{eq:poly2terms} under the wave flow.
We obtain then the following result.
\begin{theorem}
    \label{thm:selfintersections_wave}
    Let $\gamma(t) = u(\cdot, t)$ be the wave evolution of a polynomial~\eqref{eq:poly2terms}. Then
    \begin{enumerate}[(i)]
        \item The self-intersections of $\gamma(t)$, if any, lie along the directions of the powers of $\e{ \frac{2 \pi i a}{b-a}}$ for all $t>0.$
        \item On every such line, the number of self-intersections is the number of roots of 
         \begin{equation}
            \label{eq:selfint_wave}
            U_{a-1}(s) = \mp \frac{c_b}{c_a} \phi(t) U_{b-1}(s)
         \end{equation}
                        as rational function in $s \in [-1, 1]$, and have length $$|c_a \cos (2 \pi c a t) T_a(s) \mp c_b \cos (2 \pi c b t ) T_b(s)|.$$
    \end{enumerate}
\end{theorem}
\begin{proof}
    Apply Theorem~\ref{thm:selfintersections} for $v = c_a \cos (2 \pi c a t)$ and $w=c_b \cos (2 \pi c b t).$
\end{proof}
\begin{example}
    Consider the evolution $\gamma(t)$ of the curve defined by the polynomial $p(z) = vz^a + wz^b.$ If $a \nmid b$,  we see from Lemmas \ref{lemma:psi} and \ref{lemma:phi} that the range of $\psi$ and the range of $\phi$ are both $\RR$, so we can say that for all $t \in [0, T]$ there is $s \in [-1, 1]$ satisfying~\eqref{eq:selfint_wave}.
    To construct a family of curves that has no self-intersections at all times $t$ one can pick $a|b$ and weights $c_a, c_b$ such that $\frac{c_b}{c_a} < \frac{a^2}{b^2}.$ In this case, the ranges
    \[
    \frac{c_b}{c_a} \phi([0, T]) = \biggl [-\frac {b c_b}{a c_a}, \frac{b c_b}{a c_a} \biggr ]    
    \]
    and 
    \[
    \psi([0, 1]) = \biggl (-\infty, \frac a b \biggr ] \cup \biggl [\frac a b, +\infty \biggr )    
    \]
    will never intersect. A numerical example of such a situation is displayed in Figure~\ref{fig:evolution_noselfint}.
\end{example}
\begin{figure}
\begin{center}  
    \includegraphics[scale=0.2]{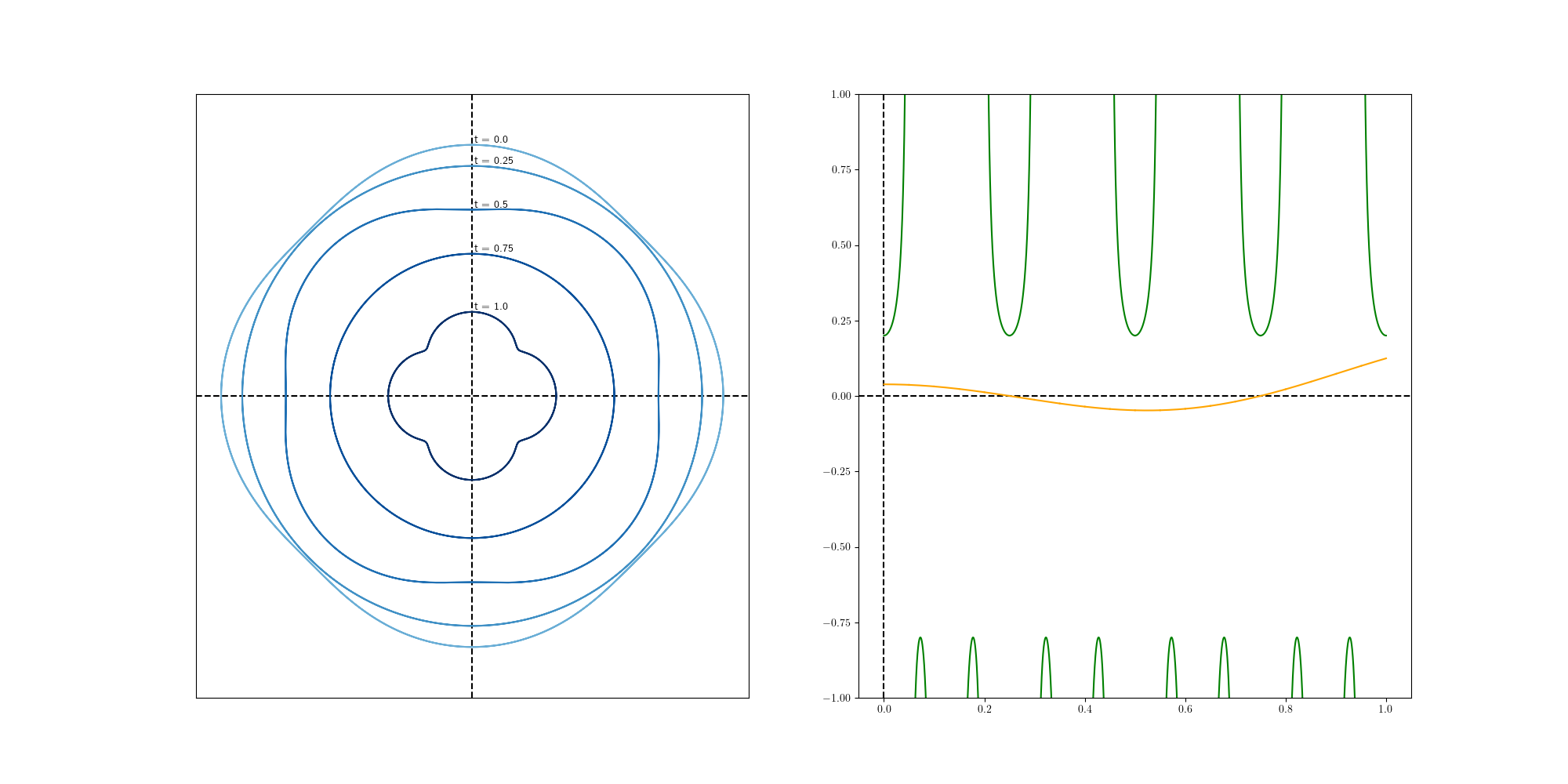}
    \caption{On the left, the evolution of $p(z) = 26 z^2 + z^{10}$ where self-intersections are absent for all times. On the right, the plots of $\psi$ and $\phi$ that never intersect}
    \label{fig:evolution_noselfint}
\end{center}
\end{figure}
\begin{example}\label{example:noselfint}
    There can be cases where there are no self-intersections at $t=0$, but for positive $t.$ Take $p(z) = nz + z^n$ with $n$ odd, that, as seen in Example~\ref{example:cusp} has $n-1$ cusps, so no self-intersections.
    We notice that for $t_k = \frac 1 {4ct} + \frac{k}{2ct}$ the function $\phi$ reaches $\pm 1.$
    On the other hand, the ranges of $\frac 1 n \phi$ and $\psi$ intersect exactly in
    \[
    \frac 1 n \phi([0, T]) \cap \psi([0, 1]) = [-1, 1] \cap ( (-\infty, -1] \cup [1, +\infty)) = \{ -1, 1\} 
    \]
where we have used that $n$ is odd together with Lemmas~\ref{lemma:psi}~and~\ref{lemma:phi}. A plot is shown in Figure~\ref{fig:evolution_withselfint}.
\end{example}
\begin{figure}
\begin{center}   
    \includegraphics[scale=0.5]{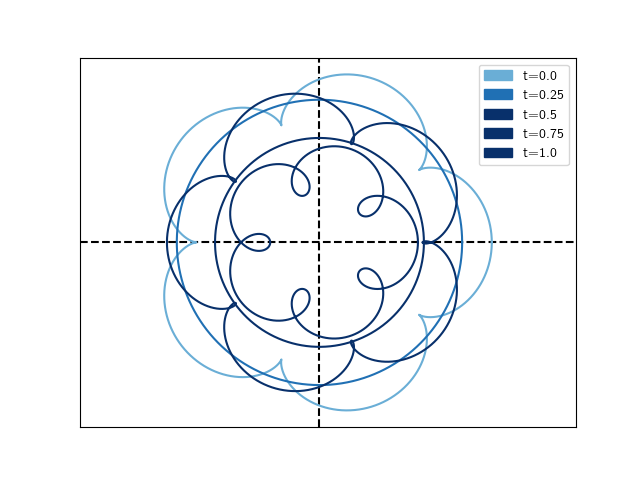}
    \caption{Evolution of $p(z) = 6z+z^6$ with $c=1/6$ and period $T=6.$ For $t=0$, the rosette has five cusps and self-intersections occur, and disappear, at later times $0<t \leq 1.$}
    \label{fig:evolution_withselfint}
  \end{center}
 \end{figure}
\begin{example} \label{ex:evolution_2z5z2}
    Let us consider the evolution of the polynomial $p(z)=z^2 + 2z^5$ that is given by
    \[
    u(x, t) = \cos (4 \pi c t) e^{2 \pi i 2x} + 2 \cos(10 \pi c t) e^{2 \pi i 5 x}.
    \] 
    If we pick $c=1/10$, we obtain a period $T=1$, and we see that the solutions of $|\phi|=c_b/c_a$ lie at $t\approx .35$ and $t \approx .62$ as we can see in Figure~\ref{fig:plotphiz22z5}. Moreover, another remarkable time is $t=.5$ where $\cos 4 \pi c t = 0$ making the curve degenerate to a circle traced twice. 
    The shape of $u(\cdot, t)$ for several values of $t < 1$ are shown in Figure~\ref{fig:plotevolz22z5}.
    According to Theorem~\ref{thm:winding2terms}, the winding number concerning zero is then $2$ for $ t \in (.35, .62)$ and $5$ outside this interval.
    Note also that, given the similarity of the equation for singularities $u_x(x, t) = 0$ with the one for self-intersection $g(x, s) = 0$, we also note that their solutions $(x, t)$ and $(x, s)$ must have their $x$-component along the same directions.
\begin{figure}
\begin{center}   
    \includegraphics[scale=0.5]{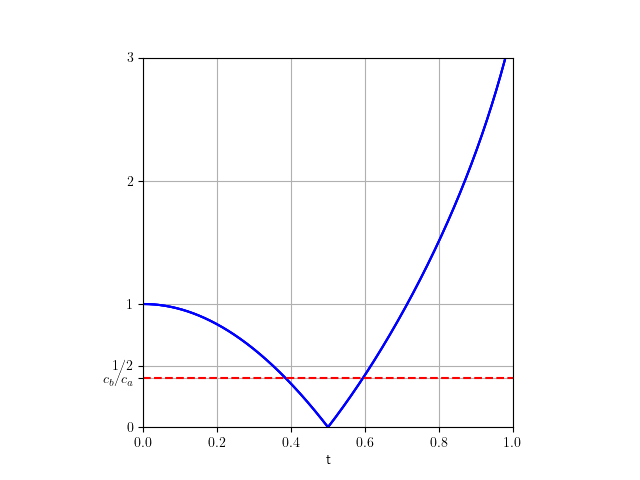}
    \caption{Plot of $|\phi| = |\cos(\pi t)|/|\cos(.4 \pi t)|$ for $0<t<1.$ The function vanishes at $t=.5$ and attains the value $c_b/c_a=.4$ at $t \approx .35$ and $t \approx .62$}
    \label{fig:plotphiz22z5}
\end{center}
\end{figure}

  \begin{figure}
  \begin{center}     
    \includegraphics[scale=0.5]{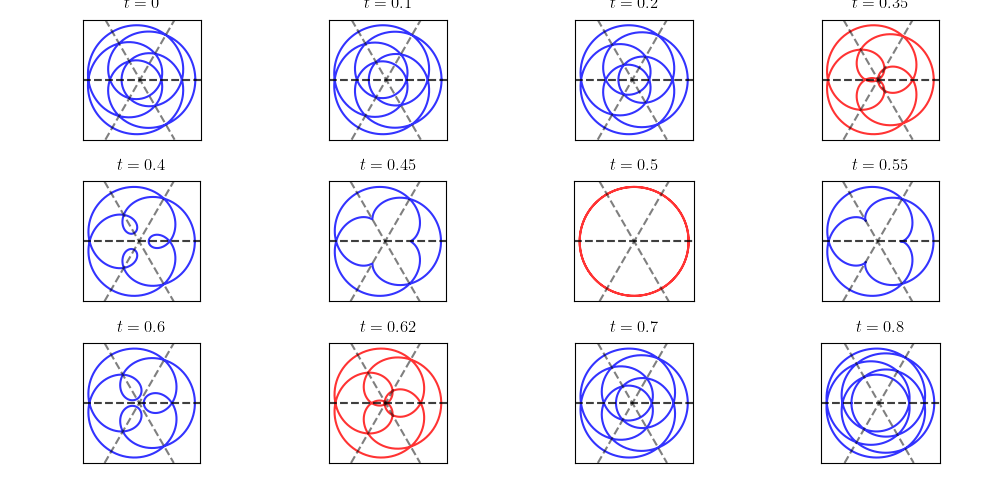}
    \caption{Time evolution of $z^2 + 2z^5$ for $0<t<1.$ The curve degenerates to a circle traced twice for $t=0.5$ and passes through zero (blue dot) at $t \approx .35$ and $t \approx .62.$ The self-intersections as well as the cusps lie along the lines given by powers of $e^{\frac 4 3 \pi i}$ at all times (dashed lines) and, for some value $t \in (.60, .62)$, the former are not simple (that is, they have the same tangent)}
    \label{fig:plotevolz22z5}
    \end{center}
  \end{figure}
\end{example}

%%%%%%%%%%%%%%%%%%%%%%%%%%%%%
%%%%%%%%%%%%%%%%%%%%%%%%%%%%%%
\section{Conclusion}

We studied the symmetry groups and winding numbers of planar curves obtained as images of the unit circle under weighted sums of exponentials as well as  their evolution under the wave equation.

Example \ref{ex:evolution_2z5z2} illustrates the main results of our paper. Starting from a given polynomial we can determine the symmetry group of the curve defined as the image of the unit circle by analysing the exponents as well as the coefficients of the polynomial; see also Figure \ref{fig:symmetry25} for a different example. In the second step, we can study the evolution of the curve under the wave equation. As seen in Figure \ref{fig:plotevolz22z5}, the characteristics of the curve can change in various ways along the evolution. Apart from theoretical interest, we see our main contribution in providing a general and practical framework within which planar curves defined as the image of the unit circle can be easily studied and within which examples of curves with prescribed properties can be easily constructed, as seen for example in Figures \ref{fig:asymmetric} and \ref{fig:multiple}.

\subsection*{Acknowledgement}
We thank the anonymous referees for their careful reading of an earlier version of the manuscript and for their constructive feedback.

The work of the first author is supported by FCT-Funda\c{c}\~{a}o para a Ci\^{e}ncia e a Tecnologia (Portugal) under project UIDB/04561/2020

%%% REFERENCES %%%
{\small\bibliography{cimart}}
% Please, do not change the above line and do not insert your references
% into this file.  Instead, insert your references into the cimart.bib file.
% See cimart.bib for further instructions.

\EditInfo{July 15, 2024.}{ December 8, 2024.}{Lenny Fukshansky.}
\end{document}